\documentclass[12pt, reqno]{amsart}
\usepackage{amsmath, amsthm, amscd, amsfonts, amssymb, graphicx, color}
\usepackage[bookmarksnumbered, colorlinks, plainpages]{hyperref}
\hypersetup{colorlinks=true,linkcolor=red, anchorcolor=green, citecolor=cyan, urlcolor=red, filecolor=magenta, pdftoolbar=true}

\newtheorem{theorem}{Theorem}[section]
\newtheorem{lemma}[theorem]{Lemma}
\newtheorem{prop}[theorem]{Proposition}
\newtheorem{cor}[theorem]{Corollary}
\theoremstyle{definition}

\theoremstyle{remark}
\newtheorem{remark}[theorem]{\bf{Remark}}
\numberwithin{equation}{section}
\allowdisplaybreaks
\begin{document}

\title[ Sharper bounds for the numerical radius   ]  
{ Sharper bounds for the numerical radius of $\MakeLowercase{n} \times \MakeLowercase{n}$ operator matrices II    }

\author[P. Bhunia ]{  Pintu Bhunia  }

\address{Pintu Bhunia, Department of Mathematics, Indian Institute of Science, Bengaluru 560012, Karnataka, India}
\email{pintubhunia5206@gmail.com; pintubhunia@iisc.ac.in}

\thanks{Dr. Pintu Bhunia  would like to thank SERB, Govt. of India for the financial support in the form of National Post Doctoral Fellowship (N-PDF, File No. PDF/2022/000325) under the mentorship of Prof. Apoorva Khare}

\subjclass[2020]{47A12, 47A30, 15A60, 47A63}
\keywords{Numerical radius, Operator norm, Operator matrix, Contraction operator}

\date{}
\maketitle
\begin{abstract}
 Let $A=[A_{ij}]$ be an $n\times n$ operator matrix where each $A_{ij}$ is a bounded linear operator on a complex Hilbert space $\mathcal{H}$. With other numerical radius bounds via contraction operators, we show that 
    $w(A) \leq  w(\Tilde{A}),$
     where $\Tilde{A}=[a_{ij}]$ is an $n\times n$ complex matrix with
     \begin{eqnarray*}
         a_{ij}=\begin{cases}
        w(A_{ii}) \quad \text{if } i=j\\
\underset{0\leq t \leq 1}{\min}     \left\| |A_{ij}|^{2t} +   |A_{ji}^*|^{2t}  \right\|^{1/2} 
        \left\| |A_{ij}^*|^{2(1-t)}+ |A_{ji}|^{2(1-t)}  \right\|^{1/2}  \quad \text{if } i< j\\
       0 \quad \text{if } i> j.
    \end{cases}
     \end{eqnarray*}
     This bound refines the well known bound $w(A) \leq  w(\hat{A}),$
     where $\hat{A}=[\hat{a}_{ij}]$ is an $n\times n$ matrix with 
$\hat{a}_{ij}= w(A_{ii}) $ \text{if } $i=j$ and $\hat{a}_{ij}= \|A_{ij}\|  $  \text{if } $i\neq  j$ [Linear Algebra Appl. 468 (2015), 18--26]. 
 We deduce that if $A$, $B$ are bounded linear operators on $\mathcal{H},$ then
    \begin{eqnarray*}
        w\left(\begin{bmatrix}
			0&A\\
			B&0
		\end{bmatrix}\right) \leq  \frac12 \left\| |A|^{2t} +   |B^*|^{2t}  \right\|^{1/2} 
        \left\| |A^*|^{2(1-t)}+ |B|^{2(1-t)}  \right\|^{1/2}  \quad \text{for all } t\in [0,1].
    \end{eqnarray*}
     Further by applying the numerical radius bounds of operator matrices, we deduce some numerical radius bounds for a single operator, the product  of two operators, the commutator of operators. 
We show that if $A$ is a bounded linear operator on $\mathcal{H},$ then
$w(A) \leq \frac12  \|A\|^t  \left\|  |A|^{1-t}+|A^*|^{1-t}  \right\| \quad \text{for all } t\in [0,1],$ which refines as well as generalizes the existing ones.

\end{abstract}



\section{Introduction}

\noindent  
Let $\mathcal{B}(\mathcal{H})$ denote the $C^*$-algebra of all bounded linear operators on a complex Hilbert space  $\mathcal{H}$.
If $\mathcal{H}$ is an $n$-dimensional Hilbert space, then $\mathcal{B}(\mathcal{H})$ is identified with $\mathcal{M}_n(\mathbb{C}),$ the set of all $n\times n$ complex matrices.
For $A\in \mathcal{B}(\mathcal{H})$, let $|A|=(A^*A)^{1/2}$ and $|A^*|=(AA^*)^{1/2},$ where $A^*$ is the Hilbert adjoint of $A.$ Let $ r(A), \|A\|$ and $w(A)$ denote the spectral radius, the operator norm and the numerical radius of $A,$ respectively. Recall that 
$\|A\|=\sup\left\{ \| Ax\|: x\in \mathcal{H}, \|x\|=1  \right\}$ and
$w(A)=\sup\left\{ |\langle Ax,x\rangle|: x\in \mathcal{H}, \|x\|=1  \right\}.$
The numerical radius, $w(\cdot): \mathcal{B}(\mathcal{H})\to \mathbb{R}$ defines a norm on $\mathcal{B}(\mathcal{H})$ and it satisfies $\frac12 \|A\|\leq w(A)\leq \|A\|,$ for every $A\in \mathcal{B}(\mathcal{H}).$
Here $w(A)=\|A\|=r(A)$ if $A$ is normal and $w(A)=\frac12 \|A\|$ if $A^2=0.$ A complete characterization of $w(A)=\frac12 \|A\|$ is given in \cite{Bhunia_LAA-C}.
Analogous to the operator norm, the numerical radius satisfies $w(A^n) \leq w^n(A)$, for every positive integer $n$, and $w(A^n) = w^n(A)$ if $A$ is normal. For further readings on the numerical radius, we refer to \cite{book, book2}. 
It is well known that (\cite[p. 44]{Horn} and \cite{PMSC}) if $A=[a_{ij}]$ is an $n\times n$ complex matrix such that $a_{ij}\geq 0$ for all $i,j=1,2,\ldots,n$, then 
	\begin{eqnarray}\label{pp0}
	    w(A)=\frac12 w\left({A+A^*}\right)=\frac12r\left({A+A^*}\right).
	\end{eqnarray} 
Therefore, the spectral radius monotonicity of matrices with non-negative entries (see in \cite[p. 491]{Horn2} implies that if $A=\begin{bmatrix}
	a_{ij}
\end{bmatrix}$ and $B=\begin{bmatrix}
b_{ij}
\end{bmatrix}$ are $n\times n$ complex matrices with $0\leq a_{ij}\leq b_{ij}$ for all $i,j=1,2,\ldots,n,$ then $w(A) \leq w(B).$

\noindent Let $A_{ij}\in \mathcal{B}(\mathcal{H})$ for all $i,j=1,2,\ldots,n.$ Then  $A=[A_{ij}]\in \mathcal{B}\left(\oplus_{i=1}^n\mathcal{H}\right)$ is an $n\times n$  operator matrix. The operator matrices, a useful tool to study Hilbert space operators, have been studied  over the years, see \cite{Halmos}. In 1995, Hou and Du \cite{Hou_1995} proved that 
\begin{eqnarray}\label{p7}
	w\left(A \right) \leq w(\Tilde{A}) \quad \text{where } \Tilde{A}=[a_{ij}]\in \mathcal{M}_n(\mathbb{C}) \text{ with } a_{ij}=\|A_{ij}\|.
\end{eqnarray}
Then, in 2015, Abu-Omar and Kittaneh \cite{Abu_LAA_2015} provided a refinement of \eqref{p7}, namely, $w\left(A \right) \leq w(\Tilde{A}) \quad \text{where } \Tilde{A}=[a_{ij}]\in \mathcal{M}_n(\mathbb{C}) \text{ with }$
\begin{eqnarray}\label{p8}
	 a_{ij}=  \begin{cases}
	    w(A_{ii}) \quad \text{ if } i=j\\
     \|A_{ij}\| \quad \text{ if } i\neq j.
	\end{cases}
\end{eqnarray}
Recently, in 2024,  Bhunia \cite[Corollary 2.4]{Bhunia-AdM} developed a refinement of \eqref{p8}, namely, $w\left(A \right) \leq w(\Tilde{A}) \quad \text{where } \Tilde{A}=[a_{ij}]\in \mathcal{M}_n(\mathbb{C}) \text{ with }$
\begin{eqnarray}\label{p--0}
      a_{ij}=  \begin{cases}
	    w(A_{ii}) \quad \text{ if } i=j\\
     \left\| |A_{ij}| +|A_{ji}^*| \right \|^{1/2} \left\| |A_{ji}|+|A_{ij}^*| \right\|^{1/2} \quad \text{ if } i< j\\
     0 \quad \text{ if } i> j.
	\end{cases}
\end{eqnarray}
In this paper, we obtain a refinement of \eqref{p--0}. Among other numerical radius bounds, we prove that  $w\left(A \right) \leq w(\Tilde{A}) \quad \text{where } \Tilde{A}=[a_{ij}]\in \mathcal{M}_n(\mathbb{C}) \text{ with }$
\begin{eqnarray*}\label{}
     a_{ij}=\begin{cases}
        w(A_{ii}) \quad \text{if } i=j\\
      \Big(  \left\| |A_{ij}|^{1/2} |K_{ij}| |A_{ij}|^{1/2} + |A_{ji}^*|^{1/2} |K_{ji}^*| |A_{ji}^*|^{1/2}\right\|^{1/2} \\ 
       \quad \times \left\| |A_{ij}^*|^{1/2} |K_{ij}^*| |A_{ij}^*|^{1/2}+ |A_{ji}|^{1/2} |K_{ji}| |A_{ji}|^{1/2}  \right\|^{1/2} \Big) \quad \text{if } i< j\\
       0 \quad \text{if } i> j,
    \end{cases}
\end{eqnarray*}
where $K_{ij}$, for all $i,j=1,2,\ldots, n$, are contractions satisfy $A_{ij}= |A_{ij}^*|^{1/2} K_{ij} |A_{ij}|^{1/2}$. We also deduce the numerical radius bounds for certain $2\times 2$ operator matrices.
   Applying the numerical radius bounds for operator matrices, we deduce  some new numerical radius bounds for a single operator, the product of two operators, the sum of the product of two pairs of operators and the commutator of operators. We consider computational examples to illustrate the results.
   

\section{Main results}

First we develop upper bounds for the numerical radius of $n\times n$ operator matrices which improve the bounds \eqref{p8} and \eqref{p--0}.  To achieve the results we need the following lemma.

\begin{lemma} \cite[Proposition 3.4]{Bhunia-S} \label{lem-1}
Let $A\in \mathcal B\left( \mathcal H \right)$ and let  $f,g$ be non-negative continuous functions on $\left[ 0,\infty  \right)$ such that $f\left( \lambda \right)g\left( \lambda \right)=\lambda$, $\lambda\geq 0$.  Then there exists a contraction $K$ (i.e., $\|K\|\leq 1$), where  $A= g(|A^*|) K f(|A|)$,  such that
$  \left| \left\langle Ax,y \right\rangle  \right|^2 \le {  \left\langle f\left( \left| A \right| \right){{\left| {{K}^{}} \right|}}f\left( \left| A \right| \right)x,x \right\rangle \left\langle g\left( \left| {{A}^{*}} \right| \right){{\left| K^* \right|}}g\left( \left| {{A}^{*}} \right| \right)y,y \right\rangle  } $ 
$ \text{for all } x,y\in \mathcal H. $ 
\end{lemma}

We now obtain an upper bound for the numerical radius of $n\times n$ operator matrices which improves \eqref{p8}.

\begin{theorem}\label{th1}
    Let $A=[A_{ij}]$ be an $n\times n$ operator matrix where $A_{ij}\in \mathcal{B}(\mathcal{H})$ for all $i,j=1,2,\ldots, n.$ Let $f_{ij}$ and $g_{ij}$ be non-negative continuous functions on $\left[ 0,\infty  \right)$ such that $f_{ij}\left( \lambda \right)g_{ij}\left( \lambda \right)=\lambda$, $\lambda\geq 0$. Then there exist contractions $K_{ij}$, where $A_{ij}= g_{ij}(|A_{ij}^*|) K_{ij} f_{ij}(|A_{ij}|)$, such that
    $w(A) \leq  w(\Tilde{A}),$
    where $\Tilde{A}=[a_{ij}]\in \mathcal{M}_n(\mathbb{C})$  with
    $$a_{ij}=\begin{cases}
        w(A_{ii}) \quad \text{if } i=j\\
       \left\| f_{ij}(|A_{ij}|) |K_{ij}| f_{ij}(|A_{ij}|)\right\|^{1/2} \left\| g_{ij}(|A_{ij}^*|) |K_{ij}^*| g_{ij}(|A_{ij}^*|)\right\|^{1/2} \quad \text{if } i\neq j.
    \end{cases}$$
    \end{theorem}

\begin{proof}
    Let  $x=  \begin{bmatrix}
	x_1\\
	x_2\\
	\vdots\\
	x_n
\end{bmatrix}\in \oplus_{i=1}^{n} \mathcal{H}$ with $\|x\|=1$,  i.e., $\|x_1\|^2+\|x_2\|^2+\ldots+\|x_n\|^2=1.$
Then, using Lemma \ref{lem-1}, we have
\begin{eqnarray}\label{p1}
	 |\langle Ax,x\rangle|
 & = & \left| \sum_{i,j=1}^{n}  \langle A_{ij}x_j,x_i\rangle \right|\notag\\
	& \leq &  \sum_{i,j=1}^{n}  \left| \langle A_{ij}x_j,x_i\rangle \right|\notag
	= \sum_{i=1}^{n}  \left| \langle A_{ii}x_i,x_i\rangle \right|+\sum_{\underset{i\neq j}{i,j=1}}^{n}  \left| \langle A_{ij}x_j,x_i\rangle \right|\notag\\
		&\leq & \sum_{i=1}^{n}  \left| \langle A_{ii}x_i,x_i\rangle \right| \notag \\
		 &+& \sum_ {{\underset{i\neq j}{i,j=1}}}^{n}   \langle   f_{ij}(|A_{ij}|) |K_{ij}| f_{ij}(|A_{ij}|)x_j,x_j\rangle^{1/2} \langle g_{ij}(|A_{ij}^*|) |K_{ij}^*| g_{ij}(|A_{ij}^*|)x_i,x_i \rangle^{1/2}.  \notag
\end{eqnarray}
 Therefore, we obtain
\begin{eqnarray*}
	|\langle Ax,x\rangle| &\leq& \sum_{i=1}^{n} w(A_{ii})\|x_i\|^2\\
 &+&\sum_ {{\underset{i\neq  j}{i,j=1}}}^{n} \left\| f_{ij}(|A_{ij}|) |K_{ij}| f_{ij}(|A_{ij}|)\right\|^{1/2} \left\| g_{ij}(|A_{ij}^*|) |K_{ij}^*| g_{ij}(|A_{ij}^*|)\right\|^{1/2} \|x_i\|\|x_j\|\\
	&=& \left \langle \Tilde{A} |x|, |x| \right \rangle,
\end{eqnarray*}
where $|x|=\begin{bmatrix}
	\| x_1\|\\
	 \|x_2\|\\
	  \vdots\\
	   \|x_n\|
\end{bmatrix}$ is an unit vector in $ \mathbb{C}^n$.
Therefore, $	|\langle Ax,x\rangle| \leq w(\Tilde{A})$ for all $x\in \oplus_{i=1}^n \mathcal{H}$ with $\|x\|=1.$ This gives $w(A)\leq w(\Tilde{A}).$ 
\end{proof}

Considering $f_{ij}(\lambda)=\lambda^t $ and   $ g_{ij}(\lambda)={\lambda}^{1-t}$ ($0\leq t\leq 1$) in Theorem \ref{th1}, we get 

\begin{cor}\label{cor1-1}
    Let $A=[A_{ij}]$ be an $n\times n$ operator matrix where $A_{ij}\in \mathcal{B}(\mathcal{H})$ for all $i,j=1,2,\ldots, n.$  Then there exist contractions $K_{ij}$,  where $A_{ij}= |A_{ij}^*|^{1-t} K_{ij} |A_{ij}|^t$ $(0\leq t\leq 1)$, such that
$w(A) \leq  w(\Tilde{A}),$
    where $\Tilde{A}=[a_{ij}]\in \mathcal{M}_n(\mathbb{C})$  with
    $$a_{ij}=\begin{cases}
        w(A_{ii}) \quad \text{if } i=j\\
       \left\| |A_{ij}|^t |K_{ij}| |A_{ij}|^t \right\|^{1/2} \left\| |A_{ij}^*|^{1-t} |K_{ij}^*| |A_{ij}^*|^{1-t}\right\|^{1/2} \quad \text{if } i\neq j.
    \end{cases}$$
    
    In particular, for $t=1/2,$ we have $w(A) \leq  w(\Tilde{A}),$
    where $\Tilde{A}=[a_{ij}]\in \mathcal{M}_n(\mathbb{C})$  with
    $$a_{ij}=\begin{cases}
        w(A_{ii}) \quad \text{if } i=j\\
       \left\| |A_{ij}|^{1/2} |K_{ij}| |A_{ij}|^{1/2}\right\|^{1/2} \left\| |A_{ij}^*|^{1/2} |K_{ij}^*| |A_{ij}^*|^{1/2}\right\|^{1/2} \quad \text{if } i\neq j.
    \end{cases}$$
    
    \end{cor}



\begin{remark}
    Since $|K_{ij}|\leq 1$ and $|K^*_{ij}|\leq 1$, so
    $|A_{ij}|^{1/2} |K_{ij}| |A_{ij}|^{1/2} \leq | A_{ij}|$ and $|A_{ij}^*|^{1/2} |K_{ij}^*| |A_{ij}^*|^{1/2} \leq |A_{ij}^*|.$
    Thus, $\left\| |A_{ij}|^{1/2} |K_{ij}| |A_{ij}|^{1/2}\right\|^{1/2} \left\| |A_{ij}^*|^{1/2} |K_{ij}^*| |A_{ij}^*|^{1/2}\right\|^{1/2} \leq \|A_{ij}\|.$ 
    Therefore,  the numerical radius bound \eqref{p8} follows from
    Corollary \ref{cor1-1}.
\end{remark}

We now obtain another upper bound for the numerical radius of $n\times n$ operator matrices which improves \eqref{p--0}.

\begin{theorem}\label{th2}
    Let $A=[A_{ij}]$ be an $n\times n$ operator matrix where $A_{ij}\in \mathcal{B}(\mathcal{H})$ for all $i,j=1,2,\ldots, n.$ Let $f_{ij}$ and $g_{ij}$ be non-negative continuous functions on $\left[ 0,\infty  \right)$ such that $f_{ij}\left( \lambda \right)g_{ij}\left( \lambda \right)=\lambda$, $\lambda \geq 0$. Then there exist contractions $K_{ij}$,  where $A_{ij}= g_{ij}(|A_{ij}^*|) K_{ij} f_{ij}(|A_{ij}|)$, such that
$w(A) \leq  w(\Tilde{A}),$
    where $\Tilde{A}=[a_{ij}] \in \mathcal{M}_n(\mathbb{C})$  with
    $$a_{ij}=\begin{cases}
        w(A_{ii}) \quad \text{if } i=j\\
      \Big(  \left\| f_{ij}(|A_{ij}|) |K_{ij}| f_{ij}(|A_{ij}|)+   g_{ji}(|A_{ji}^*|) |K_{ji}^*| g_{ji}(|A_{ji}^*|) \right\|^{1/2}\\ 
       \quad \times \left\| g_{ij}(|A_{ij}^*|) |K_{ij}^*| g_{ij}(|A_{ij}^*|)+ f_{ji}(|A_{ji}|) |K_{ji}| f_{ji}(|A_{ji}|) \right\|^{1/2} \Big) \quad \text{if } i< j\\
       0 \quad \text{if } i> j.
    \end{cases}$$
    \end{theorem}
    
\begin{proof}
Let  $x=  \begin{bmatrix}
	x_1\\
	x_2\\
	\vdots\\
	x_n
\end{bmatrix}\in \oplus_{i=1}^{n} \mathcal{H}$ with $\|x\|=1$,  i.e., $\|x_1\|^2+\|x_2\|^2+\ldots+\|x_n\|^2=1.$
Then, using  Lemma \ref{lem-1}, we have
\begin{eqnarray}\label{p1}
	 |\langle Ax,x\rangle|
 & = & \left| \sum_{i,j=1}^{n}  \langle A_{ij}x_j,x_i\rangle \right|\notag\\
	& \leq &  \sum_{i,j=1}^{n}  \left| \langle A_{ij}x_j,x_i\rangle \right|\notag
	= \sum_{i=1}^{n}  \left| \langle A_{ii}x_i,x_i\rangle \right|+\sum_{\underset{i\neq j}{i,j=1}}^{n}  \left| \langle A_{ij}x_j,x_i\rangle \right|\notag\\
	&=& \sum_{i=1}^{n}  \left| \langle A_{ii}x_i,x_i\rangle \right|+\sum_ {{\underset{i< j}{i,j=1}}}^{n}  \left( \left| \langle A_{ij}x_j,x_i\rangle \right|+\left| \langle A_{ji}x_i,x_j\rangle \right| \right)\notag\\
		&\leq & \sum_{i=1}^{n}  \left| \langle A_{ii}x_i,x_i\rangle \right| \notag \\
		 &+& \sum_ {{\underset{i< j}{i,j=1}}}^{n}  \Big (  \langle   f_{ij}(|A_{ij}|) |K_{ij}| f_{ij}(|A_{ij}|)x_j,x_j\rangle^{1/2} \langle g_{ij}(|A_{ij}^*|) |K_{ij}^*| g_{ij}(|A_{ij}^*|)x_i,x_i \rangle^{1/2}  \notag\\
   &+&  \langle f_{ji}(|A_{ji}|) |K_{ji}| f_{ji}(|A_{ji}|)x_i,x_i\rangle^{1/2}   \langle  g_{ji}(|A_{ji}^*|) |K_{ji}^*| g_{ji}(|A_{ji}^*|) x_j,x_j\rangle^{1/2}   \Big). \notag
\end{eqnarray}
 Therefore, from the Cauchy-Schwarz inequality, we get  
\begin{eqnarray*}
	|\langle Ax,x\rangle| &\leq& \sum_{i=1}^{n} w(A_{ii})\|x_i\|^2\\
 &+&\sum_ {{\underset{i< j}{i,j=1}}}^{n} \Big(  \left\| f_{ij}(|A_{ij}|) |K_{ij}| f_{ij}(|A_{ij}|)+   g_{ji}(|A_{ji}^*|) |K_{ji}^*| g_{ji}(|A_{ji}^*|) \right\|^{1/2} \\ 
    &&  \times  \left\| g_{ij}(|A_{ij}^*|) |K_{ij}^*| g_{ij}(|A_{ij}^*|)+ f_{ji}(|A_{ji}|) |K_{ji}| f_{ji}(|A_{ji}|) \right\|^{1/2} \Big) \|x_i\|\|x_j\|\\
	&=& \left \langle \Tilde{A} |x|, |x| \right \rangle,
\end{eqnarray*}
where $|x|=\begin{bmatrix}
	\| x_1\|\\
	 \|x_2\|\\
	  \vdots\\
	   \|x_n\|
\end{bmatrix}$ is an unit vector in $ \mathbb{C}^n$.
Therefore, $	|\langle Ax,x\rangle| \leq w(\Tilde{A})$ for all $x\in \oplus_{i=1}^n \mathcal{H}$ with $\|x\|=1$ and so $w(A)\leq w(\Tilde{A}).$  
\end{proof}

Considering $f_{ij}(\lambda)= \lambda^t, g_{ij}(\lambda)=\lambda^{1-t}$ where $i<j$ and $f_{ji}(\lambda)= \lambda^{1-t}, g_{ji}(\lambda)=\lambda^{t}$ where $i<j$ ($0\leq t\leq 1$) in Theorem \ref{th2}, we obtain 

\begin{cor}\label{cor2}
    Let $A=[A_{ij}]$ be an $n\times n$ operator matrix where $A_{ij}\in \mathcal{B}(\mathcal{H})$ for all $i,j=1,2,\ldots, n.$  Then there exist contractions $K_{ij}$,  where $A_{ij}= |A_{ij}^*|^{1-t} K_{ij} |A_{ij}|^t$ and $A_{ji}= |A_{ji}^*|^{t} K_{ji} |A_{ji}|^{1-t}$, $i<j$ $(0\leq t\leq 1)$, such that
   $w(A) \leq  w(\Tilde{A}),$
    where $\Tilde{A}=[a_{ij}]\in \mathcal{M}_n(\mathbb{C})$   with
    $$a_{ij}=\begin{cases}
        w(A_{ii}) \quad \text{if } i=j\\
      \Big(  \left\| |A_{ij}|^t |K_{ij}| |A_{ij}|^t+   |A_{ji}^*|^t |K_{ji}^*| |A_{ji}^*|^t \right\|^{1/2}\\ 
       \quad \times \left\| |A_{ij}^*|^{1-t} |K_{ij}^*| |A_{ij}^*|^{1-t}+ |A_{ji}|^{1-t} |K_{ji}| |A_{ji}|^{1-t} \right\|^{1/2} \Big) \quad \text{if } i< j\\
       0 \quad \text{if } i> j.
    \end{cases}$$
    \end{cor}

Again, considering $f_{ij}(\lambda)= \lambda^t$ and $ g_{ij}(\lambda)=\lambda^{1-t}$ for all $i\neq j$ $(0\leq t\leq 1)$ in Theorem \ref{th2}, we obtain 

\begin{cor}\label{cor3}
    Let $A=[A_{ij}]$ be an $n\times n$ operator matrix where $A_{ij}\in \mathcal{B}(\mathcal{H})$ for all $i,j=1,2,\ldots, n.$  Then there exist contractions $K_{ij}$,  where $A_{ij}= |A_{ij}^*|^{1-t} K_{ij} |A_{ij}|^t$ $(0\leq t\leq 1)$, such that
        $w(A) \leq  w(\Tilde{A})$
    where $\Tilde{A}=[a_{ij}]\in \mathcal{M}_n(\mathbb{C})$  with
    $$a_{ij}=\begin{cases}
        w(A_{ii}) \quad \text{if } i=j\\
      \Big(  \left\| |A_{ij}|^t |K_{ij}| |A_{ij}|^t+   |A_{ji}^*|^{1-t} |K_{ji}^*| |A_{ji}^*|^{1-t} \right\|^{1/2}\\ 
       \quad \times \left\| |A_{ij}^*|^{1-t} |K_{ij}^*| |A_{ij}^*|^{1-t}+ |A_{ji}|^{t} |K_{ji}| |A_{ji}|^{t} \right\|^{1/2} \Big) \quad \text{if } i< j\\
       0 \quad \text{if } i> j.
    \end{cases}$$

    In particular, for $t=1/2$, we have $w(A) \leq  w(\Tilde{A})$,
    where $\Tilde{A}=[a_{ij}]\in \mathcal{M}_n(\mathbb{C})$  with
    $$a_{ij}=\begin{cases}
        w(A_{ii}) \quad \text{if } i=j\\
      \Big(  \left\| |A_{ij}|^{1/2} |K_{ij}| |A_{ij}|^{1/2} + |A_{ji}^*|^{1/2} |K_{ji}^*| |A_{ji}^*|^{1/2}\right\|^{1/2} \\ 
       \quad \times \left\| |A_{ij}^*|^{1/2} |K_{ij}^*| |A_{ij}^*|^{1/2}+ |A_{ji}|^{1/2} |K_{ji}| |A_{ji}|^{1/2}  \right\|^{1/2} \Big) \quad \text{if } i< j\\
       0 \quad \text{if } i> j.
    \end{cases}$$
    \end{cor}



\begin{remark}\label{rem2}
    Let $A=[A_{ij}]$ be an $n\times n$ operator matrix where $A_{ij}\in \mathcal{B}(\mathcal{H})$ for all $i,j=1,2,\ldots, n.$ For any contraction $K$, we have $|K|\leq 1$ and $|K^*|\leq 1$. Therefore, Corollary \ref{cor2} and Corollary \ref{cor3} give the following bounds for the numerical radius. 

   \noindent (i) $w(A) \leq  w(\Tilde{A}),$
     where $\Tilde{A}=[a_{ij}]\in \mathcal{M}_n(\mathbb{C})$  with
    $$a_{ij}=\begin{cases}
        w(A_{ii}) \quad \text{if } i=j\\
    \left\| |A_{ij}|^{2t} +   |A_{ji}^*|^{2t}  \right\|^{1/2} 
        \left\| |A_{ij}^*|^{2(1-t)}+ |A_{ji}|^{2(1-t)}  \right\|^{1/2}  \quad \text{if } i< j\\
       0 \quad \text{if } i> j.
    \end{cases}$$

 \noindent   (ii) $w(A) \leq  w(\Tilde{A}),$
   where $\Tilde{A}=[a_{ij}]\in \mathcal{M}_n(\mathbb{C})$  with
    $$a_{ij}=\begin{cases}
        w(A_{ii}) \quad \text{if } i=j\\
       \left\| |A_{ij}|^{2t} +   |A_{ji}^*|^{2(1-t)}  \right\|^{1/2} 
        \left\| |A_{ij}^*|^{2(1-t)} + |A_{ji}|^{2t}  \right\|^{1/2}  \quad \text{if } i< j\\
       0 \quad \text{if } i> j.
    \end{cases}$$

 \noindent   (iii) 
        $w(A) \leq  w(\Tilde{A}),$
    where $\Tilde{A}=[a_{ij}]\in \mathcal{M}_n(\mathbb{C})$  with
    $$a_{ij}=\begin{cases}
        w(A_{ii}) \quad \text{if } i=j\\
     \left\| |A_{ij}| + |A_{ji}^*| \right\|^{1/2} 
        \left\| |A_{ij}^*|+ |A_{ji}|  \right\|^{1/2}  \quad \text{if } i< j\\
       0 \quad \text{if } i> j.
    \end{cases}$$
The inequalities in (ii) and (iii) are also proved in \cite{Bhunia-AdM} and these bounds refine the bound \eqref{p8}, see \cite[Remark 2.5]{Bhunia-AdM}.
\end{remark}

\begin{remark}  \label{rem12}
(i) Let $A=[A_{ij}]$ be an $n\times n$ operator matrix where $A_{ij}\in \mathcal{B}(\mathcal{H})$ for all $i,j=1,2,\ldots, n.$
Observe that if we choose $f_{ij}(\lambda)= \lambda^{t_{ij}}, g_{ij}(\lambda)=\lambda^{1-t_{ij}}$ where $i<j$ and $f_{ji}(\lambda)= \lambda^{1-t_{ji}}, g_{ji}(\lambda)=\lambda^{t_{ji}}$ where $i<j$ ($0\leq t_{ij}\leq 1$) in Theorem \ref{th2}, then  we deduce that
\noindent $w(A) \leq  w(\Tilde{A}),$
     where $\Tilde{A}=[a_{ij}]\in \mathcal{M}_n(\mathbb{C})$  with
    $$a_{ij}=\begin{cases}
        w(A_{ii}) \quad \text{if } i=j\\
    \underset{0\leq t \leq 1}{\min} \left\| |A_{ij}|^{2t} +   |A_{ji}^*|^{2t}  \right\|^{1/2} 
        \left\| |A_{ij}^*|^{2(1-t)}+ |A_{ji}|^{2(1-t)}  \right\|^{1/2}  \quad \text{if } i< j\\
       0 \quad \text{if } i> j.
    \end{cases}$$

 \noindent   (ii) Similarly, from Theorem \ref{th2} we also deduce that  $w(A) \leq  w(\Tilde{A}),$
   where $\Tilde{A}=[a_{ij}]\in \mathcal{M}_n(\mathbb{C})$  with
    $$a_{ij}=\begin{cases}
        w(A_{ii}) \quad \text{if } i=j\\
      \underset{0\leq t \leq 1}{\min}  \left\| |A_{ij}|^{2t} +   |A_{ji}^*|^{2(1-t)}  \right\|^{1/2} 
        \left\| |A_{ij}^*|^{2(1-t)} + |A_{ji}|^{2t}  \right\|^{1/2}  \quad \text{if } i< j\\
       0 \quad \text{if } i> j.
    \end{cases}$$
Clearly, bounds in (i) and (ii) of Remark \ref{rem12} are stronger than the same in (iii) of Remark \ref{rem2}, which is given in \cite[Corollary 2.4]{Bhunia-AdM}.
\end{remark}

Now, from Remark \ref{rem12} (i) (for $t=1$,  $t=0$) and using the relation \eqref{pp0}, we get 
\begin{prop}\label{prop4}
    Let $A=[A_{ij}]$ be an $n\times n$ operator matrix where $A_{ij}\in \mathcal{B}(\mathcal{H})$ for all $i,j=1,2,\ldots, n.$ Then 
    $w(A) \leq  w(\Tilde{A})$,
     where $\Tilde{A}=[{a}_{ij}] \in \mathcal{M}_n(\mathbb{C})$ with
    $${a}_{ij}=\begin{cases}
        w(A_{ii}) \quad \text{if } i=j\\
    \min \left\{ \sqrt{\left\| \frac{|A_{ij}|^{2} +   |A_{ji}^*|^{2}}{2}  \right\| }, \sqrt{\left\| \frac{|A_{ij}^*|^{2} +   |A_{ji}|^{2}}{2}  \right\| } \right\} 
          \quad \text{if } i\neq  j.
    \end{cases}$$

\end{prop}

\begin{remark}
    (i) Clearly, when $\|A_{ij}\|=\|A_{ji}\|$ for all $i,j=1,2,\ldots,n$, then both the bounds in Proposition \ref{prop4} refine the same in \eqref{p8}. To show proper refinement we consider an example. Take $n=2$ and $A_{11}=A_{22}=0$, $A_{12}=A_{21}=\begin{bmatrix}
        0&1\\
        0&0
    \end{bmatrix}$. Then the bounds in Proposition \ref{prop4} gives $w(A) \leq \frac{1}{\sqrt{2}}$, whereas the bound \eqref{p8} gives $w(A)\leq 1.$\\
    (ii) Consider $n=2$, $A_{11}=C,$ $ A_{22}=D$, $A_{12}=A$, $A_{21}=B$ in Proposition \ref{prop4}, we get
    \begin{eqnarray}\label{p---22}
        w\left(\begin{bmatrix}
			C&A\\
			B&D
		\end{bmatrix}\right) &\leq & \frac12 \left( w(C)+w(D) \right) \notag \\
  & + & \sqrt{ \frac14 \big( w(C)-w(D)\big)^2 + \frac12 \min \Big\{   \big\| |A|^{2} +   |B^*|^{2}  \big\| ,    \big\| {|A^*|^{2} +   |B|^{2}}  \big\|   \Big\} } \notag .
    \end{eqnarray}
    \noindent  (iii) Consider $C=D=0$ in (ii), we get
    \begin{eqnarray}\label{p---2}
        w\left(\begin{bmatrix}
			0&A\\
			B&0
		\end{bmatrix}\right) &\leq & \min \left\{ \sqrt{ \frac12 \left\| {|A|^{2} +   |B^*|^{2}}  \right\| },   \sqrt{ \frac12 \left\| {|A^*|^{2} +   |B|^{2}}  \right\| }  \right\}.
    \end{eqnarray}
  This bound refines the upper bound in \cite[Theorem 2.2]{Bhunia-AOFA}, namely, $w\left(\begin{bmatrix}
			0&A\\
			B&0
		\end{bmatrix}\right) \leq  \max \left\{ \sqrt{ \frac12 \left\| {|A|^{2} +   |B^*|^{2}}  \right\| },   \sqrt{ \frac12 \left\| {|A^*|^{2} +   |B|^{2}}  \right\| }  \right\}.$\\
  
\end{remark}

By putting $n=2$, $A_{11}=A_{22}=0$, $A_{12}=A$ and $A_{21}=B$ in Remark \ref{rem2} (i), we get

\begin{prop}\label{prop5}
Let $A,B\in \mathcal{B}(\mathcal{H}).$ Then
    \begin{eqnarray*}
        w\left(\begin{bmatrix}
			0&A\\
			B&0
		\end{bmatrix}\right) &\leq & \frac12 \left\| |A|^{2t} +   |B^*|^{2t}  \right\|^{1/2} 
        \left\| |A^*|^{2(1-t)}+ |B|^{2(1-t)}  \right\|^{1/2}  \quad \text{for all } t\in [0,1].
    \end{eqnarray*}
    In particular, for $t=1/2,$
    \begin{eqnarray}\label{p-adm}
        w\left(\begin{bmatrix}
			0&A\\
			B&0
		\end{bmatrix}\right) &\leq & \frac12 \left\| |A|^{} +   |B^*|^{}  \right\|^{1/2} 
        \left\| |A^*|^{}+ |B|^{}  \right\|^{1/2}.
    \end{eqnarray}
\end{prop}

\begin{remark}
   Considering $A=B$ in Proposition \ref{prop5}, we obtain an upper bound for the numerical radius of a bounded linear operator, namely,
   \begin{eqnarray}\label{p---112}
        w\left(A \right) &\leq & \frac12 \left\| |A|^{2t} +   |A^*|^{2t}  \right\|^{1/2} 
        \left\| |A^*|^{2(1-t)}+ |A|^{2(1-t)}  \right\|^{1/2}  \quad \text{for all } t\in [0,1].
    \end{eqnarray}
    Clearly, 
    $$ \frac12 \left\| |A|^{2t} +   |A^*|^{2t}  \right\|^{1/2} 
        \left\| |A^*|^{2(1-t)}+ |A|^{2(1-t)}  \right\|^{1/2} \leq \sqrt{\frac12 \left\| |A|^{2} +   |A^*|^{2}  \right\|} \quad \text{for all } t\in [0,1].$$
     Therefore, the bound \eqref{p---112} refines the bound  $w(A) \leq \sqrt{\frac12 \left\| |A|^{2} +   |A^*|^{2}  \right\|},$ proved by Kittaneh \cite{Kittaneh2005}.  In particular, for $t=1/2$ in \eqref{p---112}, we deduce the well known bound $w(A) \leq \frac12 \| |A|+|A^*| \|,$ proved by Kittaneh \cite{Kittaneh2003}.
\end{remark}

\begin{remark}
  (i)  Let $A,B\in \mathcal{B}(\mathcal{H}).$ Using the argument $\frac12 \|A+A^*\| \leq w(A)$ and the inequality \eqref{p-adm}, we obtain that 
    \begin{eqnarray}\label{p--adm}
        \frac12 \| A+B\| \leq w\left(\begin{bmatrix}
			0&A\\
			B^*&0
		\end{bmatrix}\right) &\leq & \frac12 \left\| |A|^{} +   |B|^{}  \right\|^{1/2} 
        \left\| |A^*|^{}+ |B^*|^{}  \right\|^{1/2}.
    \end{eqnarray}
\noindent (ii) Clearly, if $B=0,$ then $w\left(\begin{bmatrix}
			0&A\\
			0&0
		\end{bmatrix}\right)=\frac{1}{2}\|A\|.$ 
\noindent (iii) Clearly, if $B=A,$ then $w\left(\begin{bmatrix}
			0&A\\
			A^*&0
		\end{bmatrix}\right)=\|A\|.$ 
   \noindent (iv) If $A,B$ are positive, then
       $w\left(\begin{bmatrix}
			0&A\\
			B&0
		\end{bmatrix}\right)= \frac12 \| A+B\| .$
    (v) If $A,B$ are normal, then
    $ \| A+B\| \leq 2  w\left(\begin{bmatrix}
			0&A\\
			B^*&0
		\end{bmatrix}\right) \leq   \left\| |A|^{} +   |B|^{}  \right\|.$

\end{remark}

Applying the numerical radius bound in Remark \ref{rem12} (i), we now obtain a numerical radius bound for the sum of the product of two pairs of operators. 

\begin{theorem}\label{th3}
	Let $A,B,C,D\in \mathcal{B}(\mathcal{H}).$ Then
	\begin{eqnarray*}
		w(AB\pm CD) 
  &\leq & \frac14 \Big(\min_{t\in [0,1]} \left\| |A|^{2t}+ |B^*|^{2t} \right\| \left\|  |A^*|^{2(1-t)}+|B|^{2(1-t)} \right\| \\ 
		&& + \min_{t\in [0,1]}\left\| |C|^{2t}+ |D^*|^{2t} \right\| \left\| |C^*|^{2(1-t)}+|D|^{2(1-t)} \right\|  \Big).
	\end{eqnarray*} 
	 
\end{theorem}

\begin{proof}
	Considering $n=3$, $A_{11}=A_{22}=A_{33}=A_{23}=A_{32}=0$, $A_{12}=A$, $A_{13}=C$, $A_{21}=B$ and $A_{31}=D$ in Remark \ref{rem12} (i), we obtain 
	$w\left(\begin{bmatrix}
			0&A&C\\
			B&0&0\\
			D&0&0
		\end{bmatrix}\right) \leq w\left(\begin{bmatrix}
		0&a&c\\
		0&0&0\\
		0&0&0
	\end{bmatrix}\right),
	$
	where $a= \underset{{t\in [0,1]}}{\min} \left\| |A|^{2t}+ |B^*|^{2t} \right\|^{1/2} \left\||A^*|^{2(1-t)}+ |B|^{2(1-t)} \right\|^{1/2}$  and 
 \begin{eqnarray*}
     c&=& \min_{t\in [0,1]} \left\| |C|^{2t}+ |D^*|^{2t} \right\|^{1/2} \left\| |C^*|^{2(1-t)}+|D|^{2(1-t)} \right\|^{1/2} .
	\end{eqnarray*}
Therefore, now the desired bound follows from the proof of \cite[Theorem 2.8]{Bhunia-AdM}.
 \end{proof}

Similarly, from Remark \ref{rem12} (ii), we also obtain 
\begin{theorem}\label{th4}
	Let $A,B,C,D\in \mathcal{B}(\mathcal{H}).$ Then
	\begin{eqnarray*}
		w(AB\pm CD) 
  &\leq & \frac14 \Big(\min_{t\in [0,1]} \left\| |A|^{2t}+ |B^*|^{2(1-t)} \right\| \left\|  |A^*|^{2(1-t)}+|B|^{2t} \right\| \\ 
		&& + \min_{t\in [0,1]}\left\| |C|^{2t}+ |D^*|^{2(1-t)} \right\| \left\| |C^*|^{2(1-t)}+|D|^{2t} \right\|  \Big).
	\end{eqnarray*} 
\end{theorem}

Clearly, Theorem \ref{th4} is stronger than \cite[Theorem 2.8]{Bhunia-AdM}.
Considering $C=B$ and $D=A$ in Theorem \ref{th3}, we obtain the following numerical radius bound for the commutator of operators.

\begin{cor}\label{cor5}
	Let $A,B\in \mathcal{B}(\mathcal{H}).$ Then
	\begin{eqnarray*}
		w(AB\pm BA) 
  &\leq & \frac14 \Big( \min_{t\in [0,1]} \left\| |A|^{2t}+ |B^*|^{2t} \right\| \left\|  |A^*|^{2(1-t)}+|B|^{2(1-t)} \right\| \\ 
		&& +\min_{t\in [0,1]} \left\| |B|^{2t}+ |A^*|^{2t} \right\| \left\| |B^*|^{2(1-t)}+|A|^{2(1-t)} \right\|  \Big).
	\end{eqnarray*} 
	\end{cor}

\noindent Put $t=\frac12$, we get 
 $  w(AB \pm BA) \leq \frac12  \left\| |A|+ |B^*| \right\| \left\| |B|+ |A^*| \right\| ,$ which is also proved in \cite[Corollary 2.9]{Bhunia-AdM}.
Again, considering $C=D=0$ in Theorem \ref{th3}, we obtain the following numerical radius bound for the product of two operators.

\begin{cor}\label{cor6}
	If $A,B\in \mathcal{B}(\mathcal{H})$, then
	\begin{eqnarray*}
		w(AB) &\leq & \frac14  \left\| |A|^{2t}+ |B^*|^{2t} \right\| \left\|  |A^*|^{2(1-t)}+|B|^{2(1-t)} \right\|, \quad \text{for all $t\in [0,1].$ }
	\end{eqnarray*} 
\end{cor}
\begin{remark}
 Put $t=1$ and $t=0$, we get 
	$\max \{ w(AB), w(BA) \} \leq  \frac12 \left\|  |A|^2+ |B^*|^2 \right\|.$
 Also, for $t=\frac12$, we get
		$w(AB) \leq \frac14 \left\| |A|+ |B^*| \right\| \left\|  |A^*| +|B| \right\|,$ which is also proved in \cite[Corollary 2.10]{Bhunia-AdM}. 
\end{remark}

\begin{remark}
    (i)  Considering $B=U$, $C=V$ and $D=A$ in Theorem \ref{th3}, where $U$ and $V$ are unitary operators, we obtain 
	\begin{eqnarray*}
		w(AU\pm VA) 
  &\leq & \frac12  \min_{t\in [0,1]}  \left\| |A|^{2t}+ I \right\| \left\|  |A^*|^{2(1-t)}+I \right\|   . 
	\end{eqnarray*} 
 (ii)  Similarly, considering $B=U$, $C=V$ and $D=A^*$ in Theorem \ref{th3}, where $U$ and $V$ are unitary operators, we obtain 
	\begin{eqnarray*}
		w(AU\pm VA^*) 
  &\leq & \frac12 \min_{t\in [0,1]}  \left\| |A|^{2t}+ I \right\| \left\|  |A^*|^{2(1-t)}+I \right\|. 
	\end{eqnarray*} 
(iii) Considering $B=U$ in Corollary \ref{cor6}, where $U$ is an unitary operator, we obtain 
	\begin{eqnarray*}
		w(AU) 
  &\leq & \frac14 \min_{t\in [0,1]}  \left\| |A|^{2t}+ I \right\| \left\|  |A^*|^{2(1-t)}+I \right\|. 
	\end{eqnarray*} 
 
 \end{remark}

Finally, replacing $A$ and $B$ by $U|A|^{1/2}$ (where $A=U|A|$ is the polar decomposition) and $|A|^{1/2}$, respectively, in Corollary \ref{cor6}, we deduce the following numerical radius bound for a single bounded linear operator. 

\begin{prop}\label{prop1}
    If $A\in \mathcal{B}(\mathcal{H})$, then
	\begin{eqnarray*}
		w(A) &\leq &  \|A\|^t  \left\| { \frac12 \left(|A|^{1-t}+|A^*|^{1-t}\right) } \right\|, \quad \text{for all $t\in [0,1].$ }
	\end{eqnarray*} 
\end{prop}

\begin{remark}
(i) Using the concavity property of $f(\lambda)=\lambda^{1-t}$, $\lambda \geq 0$, we see that $  \|A\|^t  \left\|   \frac12 \left(|A|^{1-t}+|A^*|^{1-t}\right)  \right\| \leq \|A\|^t \left\|   \frac12 \left(|A|^{}+|A^*|^{}\right)  \right\|^{1-t} \leq \|A\|,$   for all $t\in [0,1].$ 
 Hence, Proposition \ref{prop1} refines the bound $w(A) \leq \|A\|$ for every $t\in [0,1].$

 \noindent (ii)  Taking $t=1/2$ in Proposition \ref{prop1}, we get 
    $w(A) \leq  \|A\|^{1/2}   \left\|   \frac12 \left(|A|^{1/2}+|A^*|^{1/2}\right)  \right\|,$
 which is also proved by Kittaneh et al. \cite{Kittaneh_LAMA_2023} and also see in \cite{Bhu23_2}. Therefore, we obtain
 $$ w(A) \leq \min_{t\in [0,1]}  \|A\|^t  \left\|   \frac12 \left(|A|^{1-t}+|A^*|^{1-t}\right)  \right\| \leq  \|A\|^{1/2}   \left\|   \frac12 \left(|A|^{1/2}+|A^*|^{1/2}\right)  \right\|. $$
To show proper improvement we consider a matrix $A=\begin{bmatrix}
     0&2&0\\
     0&0&3\\
     0&0&0
 \end{bmatrix}.$ Then we see that $$ \min_{t\in [0,1]}    \|A\|^t  \left\|   \frac12 \left(|A|^{1-t}+|A^*|^{1-t}\right)  \right\|=\frac{5}{2} < \frac{3+\sqrt6}{2}=\|A\|^{1/2}   \left\|   \frac12 \left(|A|^{1/2}+|A^*|^{1/2}\right)  \right\|. $$

 \end{remark}

\bigskip
\noindent \textbf{Data availability statements}
Data sharing not applicable to this article as no datasets were generated or analysed during the current study.\\
\textbf{Competing Interests}
The author has no relevant financial or non-financial interests to disclose.

\bibliographystyle{amsplain}

\begin{thebibliography}{99}

\bibitem{Abu_LAA_2015}  A. Abu-Omar and F. Kittaneh, Numerical radius inequalities for $n\times n$ operator matrices, Linear Algebra Appl. 468 (2015), 18--26.

\bibitem{Bhunia-AdM} P. Bhunia, Sharper bounds for the numerical radius of $n\times n$ operator matrices, Arch. Math. (Basel) (2024). https://doi.org/10.1007/s00013-024-02017-6



\bibitem{Bhu23_2} P. Bhunia,  Improved bounds for the numerical radius via polar decomposition of operators, Linear Algebra Appl. 683 (2024), 31–45.

\bibitem{PMSC} P. Bhunia, S. Jana and  K. Paul, Estimates of Euclidean numerical radius for block matrices, {Proc. Math. Sci.} (2024).
https://doi.org/10.1007/s12044-024-00788-0

\bibitem{Bhunia_LAA-C} P. Bhunia and K. Paul, {Corrigendum to ``Development of inequalities and characterization of equality conditions for the numerical radius'' [Linear Algebra Appl. 630 (2021) 306–315]}, {Linear Algebra Appl.} {679} ({ 2023}), 1--3.

\bibitem{Bhunia-S} P. Bhunia and S. Sahoo,  Schatten $p$-norm and numerical radius inequalities with applications, (2024). 
https://doi.org/10.48550/arXiv.2407.01962

\bibitem{book} P. Bhunia, S. S. Dragomir, M. S. Moslehian and K. Paul, Lectures on numerical radius inequalities, Infosys Sci. Found. Ser. Math. Sci. Springer Cham, 2022. XII+ 209 pp. ISBN: 978-3-031-13669-6; 978-3-031-13670-2

\bibitem{Bhunia-AOFA} P. Bhunia, S. Bag and K. Paul, Bounds for zeros of a polynomial using numerical radius of Hilbert space operators,
Ann. Funct. Anal. 12 (2021), no. 2, Paper No. 21, 14 pp.



\bibitem{Halmos} P. R. Halmos, A Hilbert Space Problem Book, 2nd ed., Springer, New York, 1982.

\bibitem{Horn} R. A. Horn and C. R. Johnson, Topics in Matrix Analysis, Cambridge University Press, Cambridge, 1991.

\bibitem{Horn2} R. A. Horn and C. R. Johnson, Matrix Analysis, Cambridge University Press, Cambridge, 1985.

\bibitem{Hou_1995} J. C. Hou and H. K. Du, Norm inequalities of positive operator matrices, Integral Equations Operator Theory 22 (1995) 281--294.

\bibitem{Kittaneh_LAMA_2023} F. Kittaneh, H. R. Moradi and M. Sababheh, Sharper bounds for the numerical radius, Linear Multilinear Algebra, (2023).  https://doi.org/10.1080/03081087.2023.2177248

\bibitem{Kittaneh2005}  F. Kittaneh, Numerical radius inequalities for Hilbert space operators, Studia Math. 168 (2005), no. 1, 73--80.

\bibitem{Kittaneh2003} F. Kittaneh, Numerical radius inequality and an estimate for the numerical radius of the Frobenius companion matrix, Studia Math. 158 (2003), no. 1, 11--17.

\bibitem{book2} P. Y. Wu and H-L Gau, Numerical ranges of Hilbert space operators. Encyclopedia of Mathematics and its Applications, 179. Cambridge University Press, Cambridge, 2021. xviii+483 pp. ISBN: 978-1-108-47906-6 47-02

\end{thebibliography}

\end{document}